\documentclass[11pt,thmsa,arabtex]{article}
\usepackage{eurosym}
\usepackage{amssymb}
\usepackage{amsfonts}
\usepackage{amsmath}
\usepackage{graphicx}
\usepackage{subfig}
\usepackage{hyperref}
\usepackage{cite}
\usepackage{ntheorem}
\usepackage{epstopdf}
\usepackage{graphicx}

\newtheorem*{theorem-non}{Theorem}
\newtheorem{theorem}{Theorem}

\newtheorem{definition}{Definition}[section]

\newtheorem{example}{Example}[section]

\newenvironment{proof}[1][Proof]{\noindent\textbf{#1.} }{\ \rule{0.5em}{0.5em}}

\oddsidemargin=0.5cm \numberwithin{theorem}{section}
\numberwithin{equation}{section}
\input{tcilatex}
\begin{document}

\title{Affine Factorable Surfaces in Pseudo-Galilean Space}
\author{\emph{H. S. Abdel-Aziz}\thanks{
~E-mail address:~habdelaziz2005@yahoo.com}, \emph{M. Khalifa Saad}\thanks{
~E-mail address:~mohamed\_khalifa77@science.sohag.edu.eg} \ and \emph{%
Haytham. A. Ali}\thanks{
~E-mail address:~haytham.ali88@yahoo.com} \\
$^{\star,\dag,\ddag}${\small \emph{Dept. of Math., Faculty of Science, Sohag
Univ., 82524 Sohag, Egypt}}\\
$^{\dag}$ {\small \emph{Dept. of Math., Faculty of Science, Islamic
University in Madinah, KSA}}}
\date{}
\maketitle

\textbf{Abstract. }An affine factorable surface of the second kind in the
three dimensional pseudo-Galilean space $G_{3}^{1}$ is studied depending on
the invariant theory and theory of differential equation. The first and
second fundamental forms, Gaussian curvature and mean curvature of the meant
surface are obtained according to the basic principles of differential
geometry. Also, some special cases are presented by changing the partial
differential equation into the ordinary differential equation to simplify
the solving process. The classification theorems of the considered surface
with zero and non zero Gaussian and mean curvatures are given. Some examples
of such a study are provided.


\section{Introduction}

In classical differential geometry, the problem of obtaining Gaussian and
mean curvatures of a surface in the Euclidean space and other spaces is one
of the most important problems, so we are interested here to study such a
problem for a surface known as affine factorable surface in the three
dimensional pseudo Galilean space $G_{3}^{1}$.

The geometry of Galilean Relativity acts like a \textquotedblleft
bridge\textquotedblright\ from Euclidean geometry to special Relativity. The
Galilean space which can be defined in three-dimensional projective space $%
P_{3}(R)$ is the space of Galilean Relativity \cite{1}. The geometries of
Galilean and pseudo-Galilean spaces have similarities, but, of course, are
different. In the Galilean and pseudo Galilean spaces, some special surfaces
such as surfaces of revolution, ruled surfaces, translation surfaces and
tubular surfaces have been studied in \cite{2,3,4,5,6,7,8,9,10}. For further
study of surfaces in the pseudo Galilean space, we refer the reader to \v{S}%
ip\v{u}s and Divjak's paper \cite{9}. Recall that the graph surfaces are
also known as Monge surfaces \cite{11}. In this work, we are interested here
in studying a special type of Monge surface, namely factorable surface of
second kind that is graph of the function $y(x,z)=f(x)g(z)$. Such surfaces
with $K,$ $H=const.$ in various ambient spaces have been classified (cf \cite%
{12,13,14,15,16}). Our purpose is to analyze the factorable surfaces in the
pseudo-Galilean space $G_{3}^{1}$ that is one of real Cayley-Klein spaces
(for details, see \cite{17,18,19}). There exist three different kinds of
factorable surfaces, explicitly, a Monge surface in $G_{3}^{1}$ is said to
be factorable (so-called homothetical) if it is given in one of the
following forms $\Phi _{1}:z(x,y)=f(x)g(y),$\ the first kind\ $\Phi
_{2}:y(x,z)=f(x)g(z),$ the second kind\ and\ $\Phi _{3}:x(y,z)=f(y)g(z)$ is
the third kind where $f$, $g$ are smooth functions \cite{14}. These surfaces
have different geometric structures in different spaces such as metric,
curvatures, etc.

\section{Basic concepts}

The pseudo-Galilean space $G_{3}^{1}$ is one of the Cayley-Klein spaces with
absolute figure that consists of the ordered triple $\{\omega ,f,I\}$, where 
$\omega $ is the absolute plane given by $x_{o}=0,$ in the three dimensional
real projective space $P_{3}(R)$, $f$ the absolute line in $\omega $ given
by $x_{o}=x_{1}=0$ and $I$ the fixed hyperbolic involution of points of $f$
and represented by $(0:0:x_{2}:x_{3})\rightarrow (0:0:x_{3}:x_{2})$, which
is equivalent to the requirement that the conic $x_{2}^{2}-x_{3}^{2}=0$ is
the absolute conic. The metric connections in $G_{3}^{1}$ are introduced
with respect to the absolute figure. In terms of the affine coordinates
given by $(x_{o}:x_{1}:x_{2}:x_{3})=(1:x:y:z)$, the distance between the
points $p=(p_{1},p_{2},p_{3})$ and $q=(q_{1},q_{2},q_{3})$ is defined by
(see \cite{9,20})%
\begin{equation*}
d(p,q)=\left\{ 
\begin{array}{c}
\left\vert q_{1}-p_{1}\right\vert ,\text{ \ \ \ \ \ \ \ \ \ \ \ \ \ \ \ \ \
\ \ }if\text{ }p_{1}\neq q_{1}, \\ 
\sqrt{\left\vert (q_{2}-p_{2})^{2}-(q_{3}-p_{3})^{2}\right\vert },if\text{ }%
p_{1}=q_{1}.%
\end{array}%
\right.
\end{equation*}

The pseudo-Galilean scalar product of the vectors $X=(x_{1},x_{2},x_{3})$
and $Y=(y_{1},y_{2},y_{3})$ is given by%
\begin{equation*}
\langle X,Y\rangle _{G_{3}^{1}}=\left\{ 
\begin{array}{c}
x_{1}y_{1},\text{ \ \ \ \ \ \ \ \ \ \ \ \ \ }if\text{ }x_{1}\neq 0\text{ }or%
\text{ }y_{1}\neq 0, \\ 
x_{2}y_{2}-x_{3}y_{3},\text{\ \ \ \ \ \ }if\text{ }x_{1}=0\text{ }and\text{ }%
y_{1}=0.%
\end{array}%
\right.
\end{equation*}

In this sense, the pseudo-Galilean norm of a vector $X$ is $\left\Vert
X\right\Vert =\sqrt{\left\vert X.X\right\vert }$. A vector $%
X=(x_{1},x_{2},x_{3})$ is called isotropic (non-isotropic) if $x_{1}=0$ $%
(x_{1}\neq 0)$. All unit non-isotropic vectors are of the form $%
(1,x_{2},x_{3})$. The isotropic vector $X=(0,x_{2},x_{3})$ is called
spacelike, timelike and lightlike if $x_{2}^{2}-x_{3}^{2}>0$, $%
x_{2}^{2}-x_{3}^{2}<0$ and $x_{2}=\pm x_{3}$, respectively. The
pseudo-Galilean cross product of $X$ and $Y$ on $G_{3}^{1}$ is given as
follows%
\begin{equation*}
X\wedge _{G_{3}^{1}}Y=\left\vert 
\begin{array}{ccc}
0 & -e_{2} & e_{3} \\ 
x_{1} & x_{2} & x_{3} \\ 
y_{1} & y_{2} & y_{3}%
\end{array}%
\right\vert ,
\end{equation*}

where $e_{2}$ and $e_{3}$ are canonical basis.

Let $M$ be a connected, oriented 2-dimensional manifold and $\phi
:M\rightarrow G_{3}^{1}$ be a surface in $G_{3}^{1}$ with parameters $(u,v)$%
. The surface parametrization $\phi $ is expressed by%
\begin{equation*}
\phi (u,v)=(x(u,v),y(u,v),z(u,v)).
\end{equation*}

On the other hand, we denote by $E$, $F$, $G$ and $L$, $M$, $N$ the
coefficients of the first and second fundamental forms of $\phi $,
respectively. The Gaussian $K$\ and mean $H$\ curvatures are%
\begin{equation}
K=\frac{LN-M^{2}}{EG-F^{2}},\text{ \ \ \ }H=\frac{EN+GL-2FM}{2\left\vert
EG-F^{2}\right\vert },
\end{equation}

where%
\begin{eqnarray*}
E &=&\phi _{u}^{\prime }.\phi _{u}^{\prime },\text{ \ \ \ }F=\phi
_{u}^{\prime }.\phi _{v}^{\prime },\text{ \ \ \ }G=\phi _{v}^{\prime }.\phi
_{v}^{\prime }, \\
L &=&\phi _{uu}^{\prime \prime }.n,\text{ \ \ \ }M=\phi _{uv}^{\prime \prime
}.n,\text{ \ \ \ }N=\phi _{vv}^{\prime \prime }.n,
\end{eqnarray*}

and%
\begin{equation*}
n=\frac{\phi _{u}^{\prime }\wedge \phi _{v}^{\prime }}{\left\vert \phi
_{u}^{\prime }\wedge \phi _{v}^{\prime }\right\vert }.
\end{equation*}

\section{Factorable surfaces in pseudo-Galilean space $G_{3}^{1}$}

In what follows, we consider the factorable surface of second kind in $%
G_{3}^{1}$ which can be locally written as

\begin{equation}
\phi (x,z)=(x,f(x)g(z),z).
\end{equation}

For this study it is important to consider the following definition:

\begin{definition}
An affine factorable surface in pseudo-Galilean space$\ G_{3}^{1}$ is
defined as a parameter surface $\phi (u,v)$\ which can be written as%
\begin{eqnarray}
\phi (u,v) &=&(x(u,v),y(u,v),z(u,v))  \notag \\
&=&(u,f(u)g(v+au),v)  \notag \\
&=&(x,f(x)g(z+ax),z),
\end{eqnarray}

for non zero constant $a$ and functions $f(x)$ and $g(z+ax)$ \cite{21}.
\end{definition}

Now, from (3.2) by a direct calculation, the first fundamental form with its
coefficients of $\phi $ can be given by

\begin{eqnarray*}
I &=&Edx^{2}+2Fdxdy+Gdy^{2}, \\
E &=&1,\text{ \ \ }F=0,\text{ \ \ \ \ }G=(fg^{\prime })^{2}-1,
\end{eqnarray*}

where%
\begin{equation*}
g^{\prime }=\frac{dg(z+ax)}{d(z+ax)}.
\end{equation*}

Also, the second fundamental form of $\phi $ is%
\begin{equation*}
II=Ldx^{2}+2Mdxdy+Ndy^{2},
\end{equation*}

note that%
\begin{eqnarray*}
L &=&\frac{\left( f^{\prime \prime }g+2af^{\prime }g^{\prime
}+a^{2}fg^{\prime \prime }\right) }{D}, \\
M &=&\frac{\left( f^{\prime }g^{\prime }+afg^{\prime \prime }\right) }{D},%
\text{ \ \ \ \ \ }N=\frac{fg^{\prime \prime }}{D}
\end{eqnarray*}

where%
\begin{equation*}
D(x,z)=\sqrt{1-(fg^{\prime })^{2}}.
\end{equation*}

In addition, the Gaussian and mean curvature of $\phi $ can be obtianed%
\begin{eqnarray}
K &=&\frac{f^{\prime 2}g^{\prime 2}-f^{\prime \prime }fg^{\prime \prime }g}{%
\left( 1-(fg^{\prime })^{2}\right) ^{2}}, \\
H &=&\frac{\Omega (x,z)}{2\left( 1-(fg^{\prime })^{2}\right) ^{\frac{3}{2}}},
\end{eqnarray}

where 
\begin{equation*}
\Omega (x,z)=(1-a^{2})fg^{\prime \prime }-f^{\prime \prime }g-2af^{\prime
}g^{\prime }+f^{2}f^{\prime \prime }g^{\prime 2}g+2af^{\prime
}f^{2}g^{\prime 3}+a^{2}f^{3}g^{\prime 2}g^{\prime \prime }.
\end{equation*}

A surface in $G_{3}^{1}$ is said to be \textit{isotropic minimal} (resp. 
\textit{flat}) if $H$ (resp. $K$) vanishes identically. Further, it is said
to have constant isotropic mean (resp. Gaussian) curvature if $H$ (resp. $K$%
) is a constant function on whole surface.

\section{Affine factorable surfaces with zero curvatures}

In this section, if the Gaussian and mean curvatures of (3.2) are vanished,
then we get the following main result:

\begin{theorem}
Let $\phi :I\subset R\rightarrow G_{3}^{1}$ be an affine factorable surface
of second kind in the form 
\begin{equation*}
\phi (x,z)=(x,f(x)g(z+ax),z),
\end{equation*}%
if its Gaussian curvature is zero, then the surface is one of the following
surfaces:

\item[(1)] $y(x,z)=f_{o}g(z+ax);$

\item[(2)] $y(x,z)=g_{o}f(x);$

\item[(3)] $y(x,z)=ce^{c_{5}x+c_{4}z};$

\item[(4)] $y(x,z)=\left[ (1-k)(c_{6}x+c_{7})\right] ^{\frac{1}{1-k}}\left[
\left( \frac{k-1}{k}\right) (c_{8}(z+ax)+c_{9})\right] ^{\frac{k}{k-1}}.$
\end{theorem}

\begin{proof}
If the Gaussian curvature of $\phi $ is zero, then from (3.3), we have%
\begin{equation}
f^{\prime 2}g^{\prime 2}-f^{\prime \prime }fg^{\prime \prime }g=0.
\end{equation}

To solve this equation we have the following cases to be discussed:

\textbf{Case1. }if $f^{\prime }=0,$ then $f^{\prime \prime }=0,$ $%
f=f_{o}=const.,$ then $y(x,z)=f_{o}g(z+ax)$.

\textbf{Case2. }if $g^{\prime }=0,$ then $g^{\prime \prime }=0,$ $%
g=g_{o}=const.,$ then $y(x,z)=g_{o}f(x)$.

\textbf{Case3. }if $f^{\prime }\neq 0$ and$\ g^{\prime }\neq 0,$ let 
\begin{equation*}
\left\{ 
\begin{array}{c}
u=x, \\ 
v=z+ax,%
\end{array}%
\right.
\end{equation*}

where $\partial (u,v)/\partial (x,z)\neq 0$. Then (4.1) can be written as%
\begin{equation*}
f_{u}^{2}g_{v}^{2}-ff_{uu}gg_{vv}=0,
\end{equation*}

or%
\begin{equation}
\left( \frac{df}{du}\right) ^{2}\left( \frac{dg}{dv}\right) ^{2}=f\frac{%
df_{u}}{df}\frac{df}{du}g\frac{dg_{v}}{dg}\frac{dg}{du}.
\end{equation}

From (4.2), we have%
\begin{equation*}
\frac{df}{du}\frac{dg}{dv}=f\frac{df_{u}}{df}g\frac{dg_{v}}{dg}.
\end{equation*}

Since $\frac{df}{du}\frac{dg}{dv}\neq 0$ and$\ g\frac{dg_{v}}{dg}\neq 0,$ so%
\begin{equation}
\left( \frac{f\frac{df_{u}}{df}}{f_{u}}\right) =\left( \frac{g_{v}}{g\frac{%
dg_{v}}{dg}}\right) ,
\end{equation}

let's write the last equation as follows%
\begin{equation}
\left( \frac{f\frac{df_{u}}{df}}{f_{u}}\right) =\left( \frac{g_{v}}{g\frac{%
dg_{v}}{dg}}\right) =k,\text{ \ \ }k=const.
\end{equation}

\ \ \ \ \ \ \ \ (a) If $k=1,$ then from (4.4), we have%
\begin{equation}
\frac{df_{u}}{f_{u}}=\frac{df}{f},\text{ \ \ }\frac{dg_{v}}{g_{v}}=\frac{dg}{%
g}.
\end{equation}

Solving this equation takes the form%
\begin{equation*}
f=c_{1}e^{c_{2}u},\text{ \ \ \ }g=c_{3}e^{c_{4}v},
\end{equation*}

where $c_{1},c_{2},c_{3},c_{4}$ are constants. And then%
\begin{eqnarray*}
y(x,z) &=&f(x)g(z+ax)=c_{1}e^{c_{2}x}c_{3}e^{c_{4}(z+ax)} \\
&=&c_{5}e^{c_{6}x+c_{4}z},
\end{eqnarray*}

where $c_{5}=c_{1}c_{3}$ and $c_{6}=c_{2}+ac_{4}$ are constants.

\textbf{\ \ \ \ \ } \ \ \ (b) When $k\neq 1,$ then from (4.4), we have%
\begin{equation*}
f\frac{df_{u}}{df}=kf_{u},\text{ \ \ \ }kg\frac{dg_{v}}{dg}=g_{v},
\end{equation*}%
which has the solution%
\begin{eqnarray*}
f(x) &=&\left[ (1-k)(c_{7}x+c_{8})\right] ^{\frac{1}{1-k}}, \\
g(z+ax) &=&\left[ \left( \frac{k-1}{k}\right) (c_{9}(z+ax)+c_{10})\right] ^{%
\frac{k}{k-1}}.
\end{eqnarray*}

Therefore we find that%
\begin{equation*}
y(x,z)=\left[ (1-k)(c_{7}x+c_{8})\right] ^{\frac{1}{1-k}}\left[ \left( \frac{%
k-1}{k}\right) (c_{9}(z+ax)+c_{10})\right] ^{\frac{k}{k-1}},
\end{equation*}

where $c_{7},c_{8},c_{9}$ and $c_{10}$ are constants.
\end{proof}

\begin{theorem}
For a given affine factorable surface of second kind in a three dimensional
pseudo-galilean space in the form%
\begin{equation*}
\phi (x,z)=(x,f(x)g(z+ax),z).
\end{equation*}

Let its mean curvature equal zero, then this surface will be one of the
following:

\item[(1)] $y(x,z)=f_{o}(b_{1}(z+ax)+b_{2}),$ or $\ y(x,z)=f_{o}\left( \sqrt{%
\frac{a^{2}-1}{a^{2}f_{o}^{2}}}(z+ax)+b_{3}\right) ;$

\item[(2)] $y(x,z)=g_{o}(b_{4}x+b_{5});$

\item[(3)] $y(x,z)=b_{8}(b_{6}x+b_{7}),$ \ or $\
y(x,z)=(b_{6}x+b_{7})(b_{9}(z+ax)+b_{10});$

\item[(4)] $y(x,z)=(b_{12}x+b_{13})(b_{11}(z+ax)+b_{12}),$ \ \ or $\ \ \
y(x,z)=\frac{1}{b_{11}}(b_{11}(z+ax)+b_{12}).$
\end{theorem}

\begin{proof}
If $H=0$, then from (3.4), we have%
\begin{equation}
(1-a^{2})fg^{\prime \prime }-f^{\prime \prime }g-2af^{\prime }g^{\prime
}+f^{2}f^{\prime \prime }g^{\prime 2}g+2af^{\prime }f^{2}g^{\prime
3}+a^{2}f^{3}g^{\prime 2}g^{\prime \prime }=0.
\end{equation}

This equation can be solved by introducing the following:

\ \ \ (1) If $f^{\prime }=f^{\prime \prime }=0$, then $f=f_{o}=const.,$ and
(4.6) becomes%
\begin{equation*}
(1-a^{2})fg^{\prime \prime }+a^{2}f^{3}g^{\prime 2}g^{\prime \prime }=0.
\end{equation*}

It can be written in a simple form%
\begin{equation*}
g^{\prime \prime }=0\text{ \ \ or \ \ }g^{\prime }=\sqrt{\frac{a^{2}-1}{%
a^{2}f_{o}^{2}}},
\end{equation*}

which gives the solution%
\begin{equation*}
g=b_{1}(z+ax)+b_{2}\text{ \ \ or \ \ }g=\sqrt{\frac{a^{2}-1}{a^{2}f_{o}^{2}}}%
(z+ax)+b_{3},
\end{equation*}

it is so%
\begin{equation*}
y(x,z)=f_{o}(b_{1}(z+ax)+b_{2}),
\end{equation*}

or%
\begin{equation*}
y(x,z)=f_{o}\left( \sqrt{\frac{a^{2}-1}{a^{2}f_{o}^{2}}}(z+ax)+b_{3}\right) ,
\end{equation*}

where $b_{1},b_{2}$ and$\ b_{3}$ are constants.

\ \ \ (2) When $g^{\prime }=g^{\prime \prime }=0$, then $g=g_{o}=const.,$
and (4.6) becomes%
\begin{equation*}
f^{\prime \prime }g=0,
\end{equation*}

which has the solution%
\begin{equation*}
f=b_{4}x+b_{5}.
\end{equation*}

Using what we got from solutions we can write%
\begin{equation*}
y(x,z)=g_{o}(b_{4}x+b_{5}),
\end{equation*}

where $b_{4},b_{5}$ are constants.

\ \ \ (3) When $f^{\prime \prime }=0$, this leads to $f^{\prime }=b_{6}$
which gives $f=b_{6}x+b_{7}.$ From (4.6), we have%
\begin{equation*}
(1-a^{2})fg^{\prime \prime }-2af^{\prime }g^{\prime }+2af^{\prime
}f^{2}g^{\prime 3}+a^{2}f^{3}g^{\prime 2}g^{\prime \prime }=0,
\end{equation*}

which can be written as%
\begin{equation*}
(1-a^{2})fg_{vv}-2af_{u}g_{v}+2af_{u}f^{2}g_{v}^{3}+a^{2}f^{3}g_{v}^{2}g_{vv}=0.
\end{equation*}

Differentiating this equation three times with respect to $u$, we obtain%
\begin{equation*}
g_{v}^{2}g_{vv}=0,
\end{equation*}

which gives%
\begin{equation*}
g_{v}=0\text{\ }\rightarrow \text{\ }g=b_{8},
\end{equation*}

or%
\begin{equation*}
g_{vv}=0\text{ }\rightarrow \text{\ }g=b_{9}(z+ax)+b_{10},
\end{equation*}

in light of this, we get%
\begin{equation*}
y(x,z)=b_{8}(b_{6}x+b_{7}),
\end{equation*}

or%
\begin{equation*}
y(x,z)=(b_{6}x+b_{7})(b_{9}(z+ax)+b_{10}),
\end{equation*}

where $b_{6},b_{7},b_{8},b_{9}$ and$\ b_{10}$ are constants$.$

\ \ (4) If $g^{\prime \prime }=0$, this means that $g^{\prime }=b_{11}$ $%
\rightarrow $ $g=b_{11}(z+ax)+b_{12}$ and\ then from (4.6), we obtain%
\begin{equation*}
f^{\prime \prime }g+2af^{\prime }g^{\prime }-f^{2}f^{\prime \prime
}g^{\prime 2}g-2af^{\prime }f^{2}g^{\prime 3}=0,
\end{equation*}

which can be written as%
\begin{equation*}
f_{uu}g+2af_{u}g_{v}-f^{2}f_{uu}g_{v}^{2}g-2af_{u}f^{2}g_{v}^{3}=0.
\end{equation*}

If we differentiate this equation with respect to $v$, we get%
\begin{equation*}
b_{11}f_{uu}-b_{11}^{3}f^{2}f_{uu}=0,
\end{equation*}%
\begin{equation*}
f_{uu}=0\text{ }\rightarrow \text{\ }f=b_{12}x+b_{13},
\end{equation*}

or%
\begin{equation*}
f=\frac{1}{b_{11}}.
\end{equation*}

So, we have%
\begin{equation*}
y(x,z)=(b_{12}x+b_{13})(b_{11}(z+ax)+b_{12}),
\end{equation*}

or%
\begin{equation*}
y(x,z)=\frac{1}{b_{11}}(b_{11}(z+ax)+b_{12}),
\end{equation*}

taking into cosideration $b_{11},b_{12}$ and$\ b_{13}$ are constants. This
completes the proof.
\end{proof}

\section{Affine factorable surfaces with non zero curvatures}

In this section, we describe the affine factorable surfaces of second kind
in $G_{3}^{1}$ when $K=const.\neq 0$ and $H=const.\neq 0$. So, we start as
follows:

\begin{theorem}
Let $\phi :I\subset R\rightarrow G_{3}^{1}$\ be an affine factorable surface
of second kind in $G_{3}^{1}$. Let its Gaussian curvature is non-zero
constant, then the surface takes the form:%
\begin{equation*}
y(x,z)=\left( g_{o}(z+ax)+\lambda _{2}\right) \left( \pm \frac{1}{g_{o}}%
\tanh \left[ \sqrt{K_{o}}x\mp g_{o}\lambda _{1}\right] \right) ,\text{ \ \ }%
\lambda _{1},\lambda _{2}\in R.
\end{equation*}
\end{theorem}

\begin{proof}
Let $K_{o}$ be a non-zero constant Gaussian curvature. Hence, we get%
\begin{equation}
K_{o}=\frac{f^{\prime 2}g^{\prime 2}-f^{\prime \prime }fg^{\prime \prime }g}{%
\left( 1-(fg^{\prime })^{2}\right) ^{2}}.
\end{equation}

from this equation, $K_{o}$ vanishes identically when $f$ or $g$ is a
constant function. Then $f$ and $g$ must be non-constant functions. We
distinguish two cases for eq(5.1):

\textbf{Case1. }$f^{\prime }=f_{o},$ $f_{o}\in R-\{0\},$ then from eq(5.1),
we can get polynomial equation on ($g^{\prime }$):%
\begin{equation*}
K_{o}-(2K_{o}f^{2}+f_{o}^{2})g^{\prime 2}+K_{o}f^{4}g^{\prime 4}=0,
\end{equation*}

which yields a contradiction.

\textbf{Case2.} $g^{\prime }=g_{o},$ $g_{o}\in R-\{0\}.$ Then eq(5.1) leads
to%
\begin{equation*}
f^{\prime }=\frac{\pm \sqrt{K_{o}-2K_{o}g_{o}^{2}f^{2}+K_{o}g_{o}^{4}f^{4}}}{%
g_{o}},
\end{equation*}

after solving this equation, we obtain%
\begin{equation*}
f(x)=\pm \frac{1}{g_{o}}\tanh \left[ g_{o}\sqrt{K_{o}}x\mp g_{o}\lambda _{1}%
\right] ,\text{ \ }\lambda _{1}\in R.
\end{equation*}%
\textbf{Case3. }$f^{\prime \prime }\neq 0,$ $g^{\prime \prime }\neq 0$. Then
eq(5.1) can be arranged as follows:%
\begin{equation*}
K_{o}=\frac{f^{\prime 2}g^{\prime 2}-f^{\prime \prime }fg^{\prime \prime }g}{%
\left( 1-(fg^{\prime })^{2}\right) ^{2}},
\end{equation*}

let $u=x,$ $v=z+ax$ and $\partial (u,v)/\partial (x,y)\neq 0$, we can obtain%
\begin{equation}
K_{o}=\frac{f_{u}^{2}g_{v}^{2}-f_{uu}fg_{vv}g}{\left( 1-(fg_{v})^{2}\right)
^{2}}.
\end{equation}%
The partial derivative of (5.2) with respect to $u$ and $v$ leads to a
polynomial equation%
\begin{equation}
\frac{f^{\prime }}{f^{2}f^{\prime \prime }}+\frac{3f^{\prime }f^{2}}{%
f^{\prime \prime }}g^{\prime 4}=0,
\end{equation}

which means that all coefficients must vanish, the contradiction $f^{\prime
}=0$ is obtained. Thus the proof is completed.
\end{proof}

\begin{theorem}
For a given affine factorable surface of second kind in $G_{3}^{1}$ which
has a non-zero constant mean curvature $H_{o}$. Then the following occurs:%
\begin{eqnarray*}
y(x,z) &=&f_{o}\left( \frac{\sqrt{9H_{o}^{2}-a^{4}f_{o}^{2}\lambda _{3}^{2}}%
}{3f_{o}H_{o}}(z+ax)+\lambda _{4}\right) , \\
y(x,z) &=&\left( -\frac{2H_{o}}{g_{o}}x^{2}+cx+c\right) g_{o}.
\end{eqnarray*}
\end{theorem}

\begin{proof}
From (3.4), we get%
\begin{equation*}
H_{o}=\frac{(1-a^{2})fg^{\prime \prime }-f^{\prime \prime }g-2af^{\prime
}g^{\prime }+f^{2}f^{\prime \prime }g^{\prime 2}g+2af^{\prime
}f^{2}g^{\prime 3}+a^{2}f^{3}g^{\prime 2}g^{\prime \prime }}{2\left(
1-(fg^{\prime })^{2}\right) ^{3/2}},
\end{equation*}

for solving this equation, the following two cases can be discussed:

\textbf{Case a. }$f=f_{o}$, $g^{\prime \prime }=\lambda _{3}=const.,$ we get%
\begin{equation*}
2H_{o}\left( 1-(fg^{\prime })^{2}\right) ^{3/2}=(1-a^{2})fg^{\prime \prime
}+a^{2}f^{3}g^{\prime 2}g^{\prime \prime },
\end{equation*}

let $u=x,$ $v=z+ax$ and $\partial (u,v)/\partial (x,y)\neq 0$, we can obtain%
\begin{equation}
2H_{o}\left( 1-(fg_{v})^{2}\right)
^{3/2}=(1-a^{2})fg_{vv}+a^{2}f^{3}g_{v}^{2}g_{vv},
\end{equation}

which has the partial derivative with respect to $v$ as%
\begin{equation*}
g_{v}=\frac{\sqrt{9H_{o}^{2}-a^{4}f_{o}^{2}\lambda _{3}^{2}}}{3f_{o}H_{o}}.
\end{equation*}%
Solving this equation gives%
\begin{equation*}
g=\pm \frac{\sqrt{9H_{o}^{2}-a^{4}f_{o}^{2}\lambda _{3}^{2}}}{3f_{o}H_{o}}%
(z+ax)+\lambda _{4},\text{ \ \ }\lambda _{4}\in R
\end{equation*}

then we have%
\begin{equation*}
y(x,z)=f_{o}\left( \frac{\sqrt{9H_{o}^{2}-a^{4}f_{o}^{2}\lambda _{3}^{2}}}{%
3f_{o}H_{o}}(z+ax)+\lambda _{4}\right) .
\end{equation*}

\textbf{Case b. }$g=g_{o}$, we get%
\begin{equation*}
2H_{o}=-f^{\prime \prime }g,
\end{equation*}

so, we obtain%
\begin{equation*}
f=-\frac{H_{o}}{g_{o}}x^{2}+\lambda _{5}x+\lambda _{6}.
\end{equation*}

Where $\lambda _{5},\lambda _{6}\in R.$ Then\ the proof is finised.
\end{proof}

Here, through the study, which presented on affine factorable surface of
second kind in pseudo-Galilean space $G_{3}^{1}$, we conclude with the
following important theory which relates between its mean and Gaussian
curvatures.

\begin{theorem}
Let $\phi :I\subset R\rightarrow G_{3}^{1}$ be an affine factorable surface
in three dimentional pseudo-Galilean space. The relation between its
Gaussian curvature $K$\ and its mean curvature $H$ is given by the formula%
\begin{equation}
H=A(x,z)K,
\end{equation}

where $A(x,z)=\frac{D^{3}(a^{2}fg^{\prime \prime }+2af^{\prime }g^{\prime
}+f^{\prime \prime }g)-fg^{\prime \prime }D}{f^{\prime \prime }fg^{\prime
\prime }g-f^{\prime 2}g^{\prime 2}},$ $D=\sqrt{1-(fg^{\prime })^{2}}$. When $%
D=0,$ then $\phi $\ is isotropic minimal affine factorable surfaces of
second kind.
\end{theorem}

\section{Some examples}

We illustrate several examples relating to the\ affine factorable surfaces
of second kind with zero and non zero Gaussian ($K$) and mean ($H$)
curvatures in the three dimentional pseudo-Galilean space $G_{3}^{1}$.

\begin{example}
Let us consider the affine factorable surfaces of second kind in $G_{3}^{1}$
given by

\item[(1)] $\phi :y(x,z)=8e^{6x+z},(x,z)\in \lbrack -1,1]\times \lbrack
0,2\pi ],$ (isotropic flat ($K=0$)),

\item[(2)] $\phi :y(x,z)=\sqrt{\frac{3}{4}}(2x+z)+9,(x,z)\in \lbrack
0,15]\times \lbrack -1,30],$ (isotropic minimal ($H=0$)),

\item[(3)] $\phi :y(x,z)=(10x+z)\tanh [x],(x,z)\in \lbrack -1,1],$ ($%
K=const.\neq 0$),

\item[(4)] $\phi :y(x,z)=-x^{2}+2x+1,(x,z)\in \lbrack -1,1],$ ($H=const.\neq
0 $)$.$

These surfaces can be drawn respectively as in Figs.1-4. 
\begin{figure}[h]
\centering
\includegraphics[width=8.5cm, height=7cm]{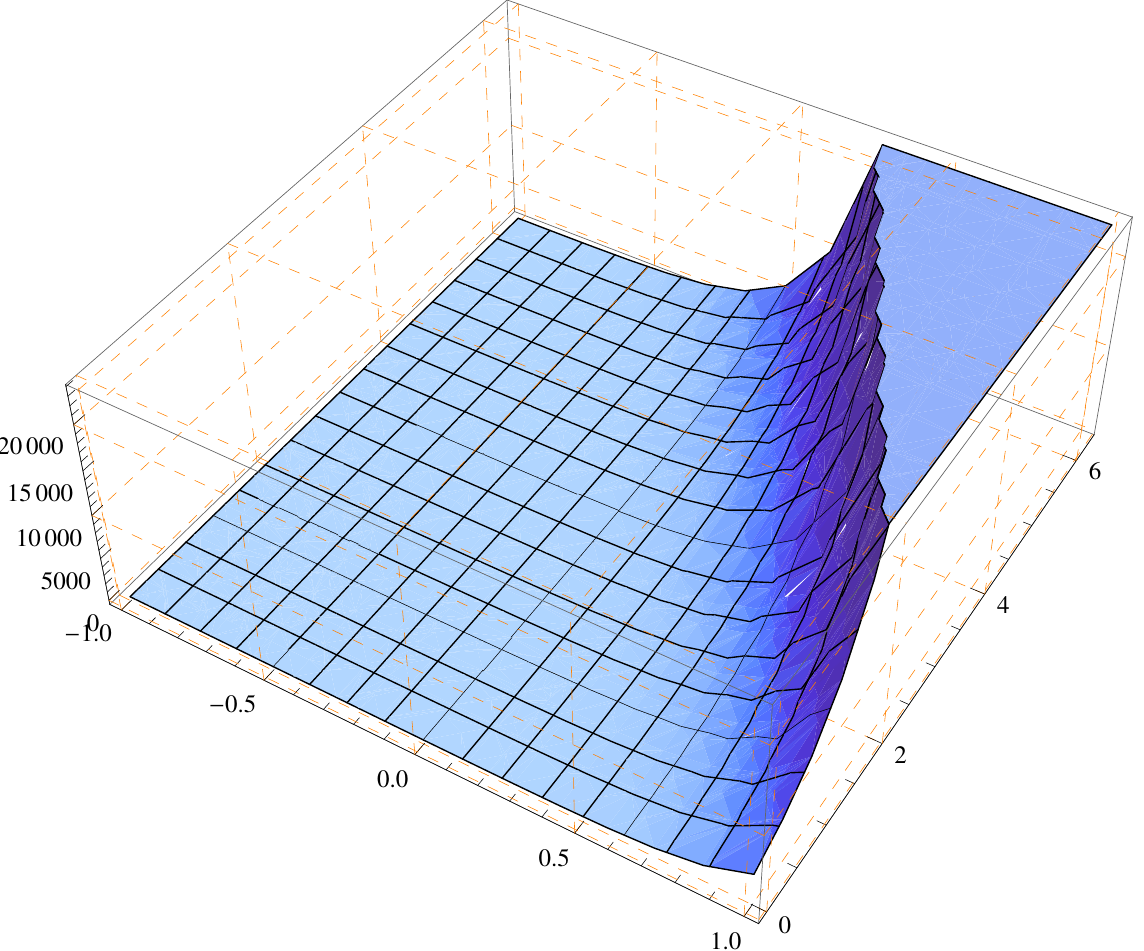}
\caption{An isotropic flat affine factorable surface of second kind.}
\label{fig:fig4}
\end{figure}
\begin{figure}[h]
\centering
\includegraphics[width=8.5cm, height=7cm]{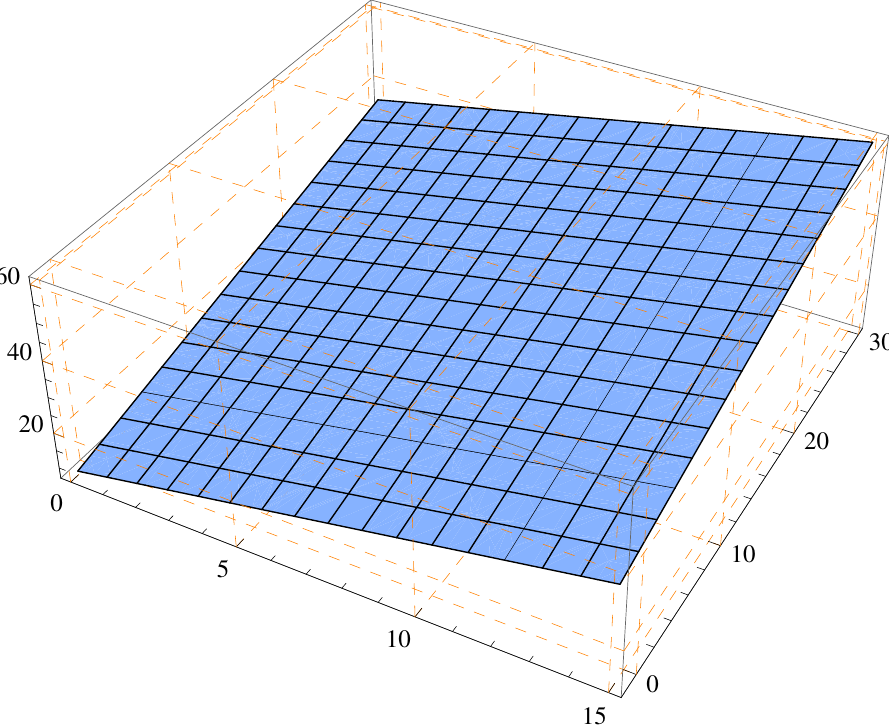}
\caption{An isotropic minimal affine factorable surface of second kind.}
\label{fig:fig4}
\end{figure}
\begin{figure}[h]
\centering
\includegraphics[width=8.5cm, height=7cm]{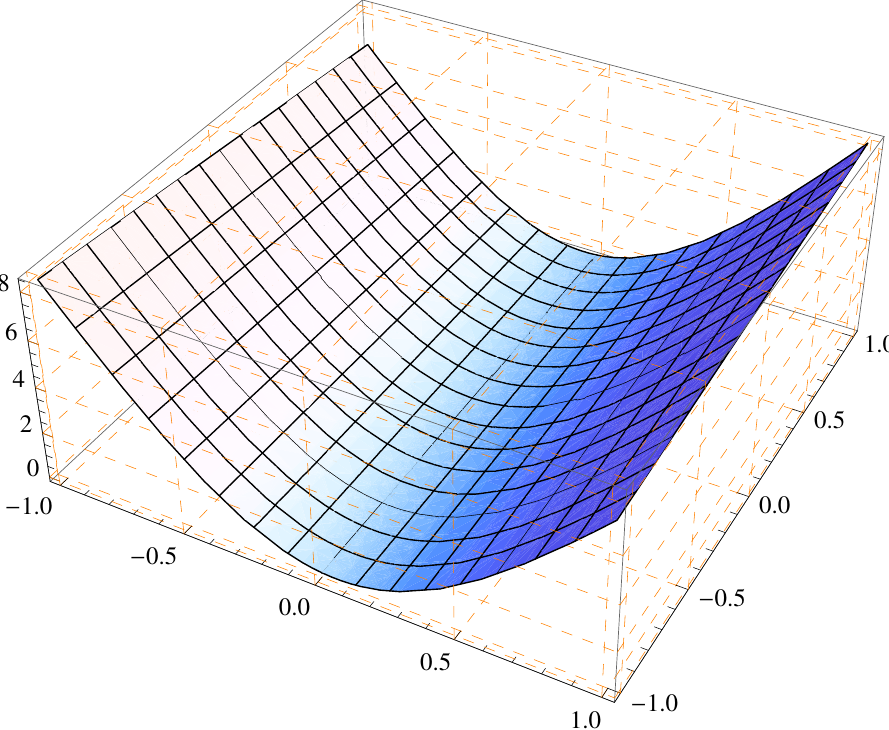}
\caption{Affine factorable surface of second kind with $K=const.\neq 0 $.}
\label{fig:fig4}
\end{figure}
\begin{figure}[h]
\centering
\includegraphics[width=8.5cm, height=7cm]{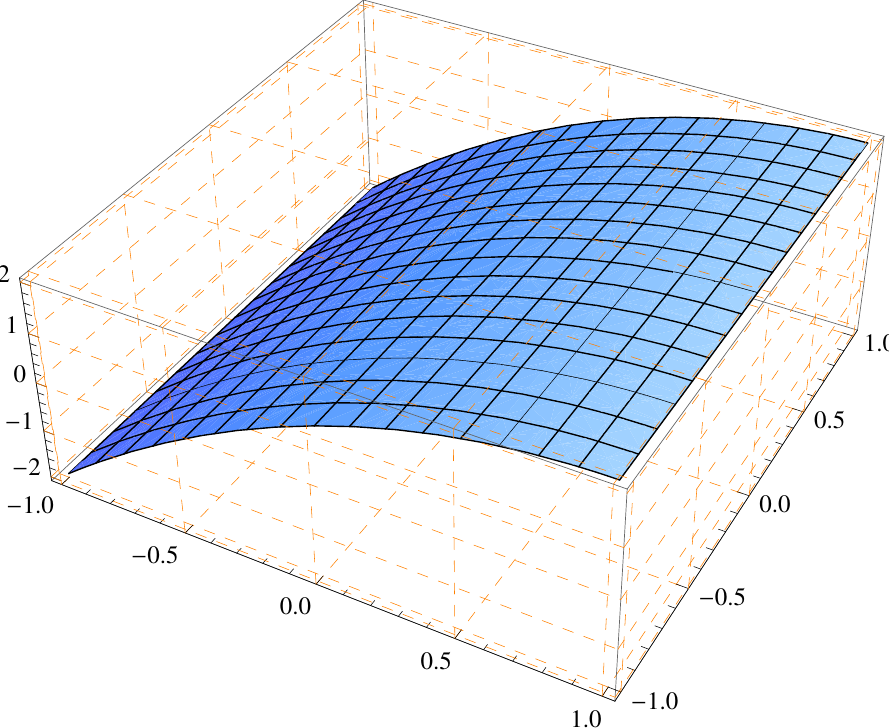}
\caption{Affine factorable surface of second kind with $H=const.\neq 0 $.}
\label{fig:fig4}
\end{figure}
\end{example}

\newpage

\section{Conclusion}

In surface theory in the field of differential geometry, especially
factorable surfaces, there are three kinds of these surfaces known as first,
second and third kind. In this paper, we are interested in studing
factorable surface of second kind which has affine form in the three
dimentional pseudo Galilean space $G_{3}^{1}$. The classification of this
surface with zero and non zero Gaussian and mean curvatures is discussed.
Also, an important relation between the curvatures of this surface is
obtained. Finally, some examples are introduced and plotted.


\begin{thebibliography}{99}
\bibitem{1} M. P. Do Carmo, Differential Geometry of Curves and Surfaces,
Prentice Hall, Englewood Cliffs, NJ, (1976).

\bibitem{2} B. Divjak and \v{Z}. M. \v{S}ipu\v{s}, Some special surfaces in
the pseudo-Galilean space, Acta Math. Hungar. 118 (2008), 209--226.

\bibitem{3} B. Divjak and \v{Z}. M. \v{S}ipu\v{s}, Minding isometries of
ruled surfaces in pseudo-Galilean space, J. Geom. 77 (2003), 35--47.

\bibitem{4} B. Divjak and \v{Z}. M. \v{S}ipu\v{s}, Special curves on ruled
surfaces in Galilean and pseudo-Galilean spaces, Acta Math. Hungar. 98
(2003), 203--215.

\bibitem{5} I. Kamenarovi\'{c}, Existence theorems for ruled surfaces in the
Galilean space $G_{3}$, Rad Hrvatske Akad. Znan. Umjet. No. 456 (1991),
183--196.

\bibitem{6} M. K. Karacan and Y. Tuncer, Tubular surfaces of Weingarten
types in Galilean and pseudo-Galilean, Bull. Math. Anal. Appl. 5 (2013),
87--100.

\bibitem{7} \u{Z}. Milin \v{S}ipu\v{s}, Ruled Weingarten surfaces in the
Galilean space, Period. Math. Hungar. 56 (2008), 213--225.

\bibitem{8} \u{Z}. Milin \v{S}ipu\v{s} and B. Divjak, Translation surface in
the Galilean space, Glas. Mat. Ser. III 46(66) (2011), 455--469.

\bibitem{9} \u{Z}. Milin \v{S}ipu\v{s} and B. Divjak, Surfaces of constant
curvature in the pseudo-Galilean space, Int. J. Math. Math. Sci. (2012), Art
ID375264, 28pp.

\bibitem{10} A. O. \"{O}\v{g}renmi\c{s} and M. Erg\"{u}t, On the Gauss map
of ruled surfaces of type II in 3-dimensional pseudo-Galilean space, Bol.
Soc. Parana. Mat. (3) 31 (2013), 145--152.

\bibitem{11} A. Gray, Modern Differential Geometry of Curves and Surfaces
with Mathematica, CRC Press LLC, (1998).

\bibitem{12} M. Bekkar and B. Senoussi, Factorable surfaces in the
three-dimensional Euclidean and Lorentzian spaces satisfying $\triangle
r_{i}=\lambda _{i}r_{i}$, J. Geom. 103 (2012), 17--29.

\bibitem{13} R. L\'{o}pez and M. Moruz, Translation and homothetical
surfaces in Euclidean space with constant curvature, J. Korean Math. Soc.
52(3) (2015), 523-535.

\bibitem{14} H. Meng and H. Liu, Factorable surfaces in Minkowski space,
Bull. Korean Math. Soc. 46(1) (2009), 155--169.

\bibitem{15} Y. Yu and H. Liu, The factorable minimal surfaces, Proceedings
of The Eleventh International Workshop on Diff. Geom. 11 (2007), 33-39.

\bibitem{16} P. Zong, L. Xiao and H. Liu, Affine factorable surfaces in
three-dimensional Euclidean space, Acta Math. Sinica Chinese Serie 58(2)
(2015), 329-336.

\bibitem{17} D. Klawitter, Clifford Algebras: Geometric Modelling and Chain
Geometries with Application in Kinematics, Springer Spektrum, (2015).

\bibitem{18} A. Onishchick and R. Sulanke, Projective and Cayley-Klein
Geometries, Springer, (2006).

\bibitem{19} I. M. Yaglom, A simple non-Euclidean Geometry and Its Physical
Basis, An elementary account of Galilean geometry and the Galilean principle
of relativity, Heidelberg Science Library. Translated from the Russian by
Abe Shenitzer. With the editorial assistance of Basil Gordon.
Springer-Verlag, New York-Heidelberg, (1979).

\bibitem{20} D. W. Yoon, Surfaces of revolution in the three dimensional
pseudo-Galilean space, Glas. Mat. Ser. III 48(68) (2013), 415--428.

\bibitem{21} U. Simon, A. Schwenk-Schellschmidt and H. Viesel, Introduction
to the affine differential geometry of hypersurfaces, Lecture Notes of the
Science University of Tokyo, Sci. Univ. Tokyo, Tokyo, (1991).
\end{thebibliography}
\end{document}